\documentclass[reqno,12pt]{amsart}
\usepackage[margin=1.1in, footskip=1cm]{geometry}
\usepackage{amsmath, amsthm, amssymb, booktabs}
\usepackage{mathrsfs}
\usepackage[usenames,dvipsnames]{color}
\pdfpagewidth=\paperwidth
\pdfpageheight=\paperheight

\newtheorem{theorem}{Theorem}[section]
\newtheorem{lemma}[theorem]{Lemma}

\theoremstyle{definition}

\theoremstyle{remark}

\numberwithin{equation}{section}

\def\d{{\partial}}

\def\le{\leqslant} \def\ge{\geqslant}

\def\d{{\,{\rm d}}}

\makeatletter
\@namedef{subjclassname@2020}{\textup{2020} Mathematics Subject Classification}
\makeatother

\begin{document}
\title[A Vinogradov subsystem]{The main conjecture for a Vinogradov subsystem}
\author[T. D. Wooley]{Trevor D. Wooley}
\address{Department of Mathematics, Purdue University, 150 N. University Street, West 
Lafayette, IN 47907-2067, USA}
\email{twooley@purdue.edu}
\date{}

\begin{abstract} When $k\ge 2$ and $s>0$, we establish that for each $\varepsilon>0$, one has
\[
\int_{[0,1)^{k-1}}\biggl| \sum_{1\le x\le X}e(\alpha_2x^2+\ldots +\alpha_k x^k)\biggr|^{2s}\d\boldsymbol \alpha 
\ll X^{s+\varepsilon}+X^{2s-(k^2+k-2)/2},
\]
confirming the main conjecture for a new family of exponential sums beyond those of Vinogradov 
(translation-dilation invariant) type. Our methods are based on very recent work of the author employing the nested 
efficient congruencing method. Consequently, analogues of our results hold also in the setting of number fields and 
function fields.
\end{abstract}

\maketitle

\section{Introduction} Our goal in this memoir is an optimal mean value estimate for exponential sums constituting  
the main conjecture for a Vinogradov subsystem. Let $k$ be a natural number with $k\ge 2$. When 
$\boldsymbol \beta =(\beta_2,\ldots ,\beta_k)\in \mathbb R^{k-1}$ and $X$ is a large positive integer, define
\begin{equation}\label{1.1}
f_k(\boldsymbol \beta;X)=\sum_{1\le x\le X}e(\beta_2x^2+\ldots +\beta_kx^k),
\end{equation}
where, as usual, we write $e(z)={\rm e}^{2\pi {\rm i}z}$. For the time being, we abbreviate $(\beta_2,\ldots ,\beta_k)$ 
to $\boldsymbol \beta$ and adopt the convention that the differential ${\rm d}\boldsymbol \beta$ denotes 
${\rm d}\beta_2\cdots \d\beta_k$. The main conclusion of this memoir is the mean value estimate described in the 
following theorem.

\begin{theorem}\label{theorem1.1}
Suppose that $k\ge 2$ and $s>0$. Then for each $\varepsilon>0$, one has
\begin{equation}\label{1.2}
\int_{[0,1)^{k-1}}|f_k(\boldsymbol \beta;X)|^{2s}\d\boldsymbol \beta \ll X^{s+\varepsilon}+X^{2s-(k^2+k-2)/2}.
\end{equation}
\end{theorem}

We pause to assess the strength of Theorem \ref{theorem1.1}. When $s\in \mathbb N$, we find via orthogonality that 
the left hand side of \eqref{1.2} counts the number of integral solutions of the system
\begin{equation}\label{1.3}
x_1^j+\ldots +x_s^j=y_1^j+\ldots +y_s^j\quad (2\le j\le k),
\end{equation}
with $1\le x_i,y_i\le X$ $(1\le i\le s)$. The solutions of \eqref{1.3} with $\{x_1,\ldots ,x_s\}=\{y_1,\ldots ,y_s\}$ make a 
contribution of size $s!X^s+O(X^{s-1})$ to the mean value on the left hand side of \eqref{1.2}. Meanwhile, the set of 
$(\beta_2,\ldots ,\beta_k)\in [0,1)^{k-1}$, with $0\le\beta_j\le (6k)^{-1}X^{-j}$, makes a contribution at least as large as 
$(6k)^{1-k}X^{-(k^2+k-2)/2}(X/2)^{2s}$. Thus, we find that
\[
\int_{[0,1)^{k-1}}|f_k(\boldsymbol \beta;X)|^{2s}\d\boldsymbol \beta \gg X^s+X^{2s-(k^2+k-2)/2},
\]
and it is apparent that Theorem \ref{theorem1.1} is essentially optimal.\par

For comparison, work \cite[Corollary 1.2]{Woo2019} utilising efficient congruencing to address relatives of 
Vinogradov's mean value theorem implies that the upper bound \eqref{1.2} holds for $s\le k(k-1)/2$. In subsequent 
investigations of the author joint with Brandes (see \cite[Theorem 1.1]{BW2017}), this range was extended to the 
interval $s\le (k^2-1)/2$. The conclusion of Theorem \ref{theorem1.1} further extends this range to the entire critical 
interval $s\le (k^2+k-2)/2$, from which the main conjecture \eqref{1.2} follows in full via the trivial estimate 
$|f_k(\boldsymbol \beta;X)|=O(X)$.\par

The conclusion of Theorem \ref{theorem1.1} provides a family of mean values, corresponding to systems of 
equations \eqref{1.3} of additive type, for which we are able to establish sharp upper bounds matching the main 
conjectures for the system. Classically, such sharp mean value estimates were available only for the special case of 
Theorem \ref{theorem1.1} in which $k=2$, a situation that concerns mean values of quadratic exponential sums. An 
addition to this elementary classical example was made only in the previous decade, when the bound
\begin{equation}\label{1.4}
\int_{[0,1)^k}\biggl| \sum_{1\le x\le X}e(\alpha_1x+\ldots +\alpha_kx^k)\biggr|^{2s}\d\boldsymbol \alpha 
\ll X^{s+\varepsilon}+X^{2s-k(k+1)/2}
\end{equation}
was established. This, the main conjecture in Vinogradov's mean value theorem, was established in work of Bourgain, 
Demeter and Guth \cite{BDG2016}, and the author \cite{Woo2016, Woo2019}. We should note that the special case of 
this estimate with $k=2$ was also classical.\par

The reader may wonder whether the methods of this memoir extend to handle mean values of exponential sums of 
the shape \eqref{1.1} in which the missing linear term is replaced by a missing term of higher degree. Thus, one 
might consider the exponential sum
\[
f_{k,d}(\boldsymbol \beta;X)=\sum_{1\le x\le X}e\biggl( \sum_{\substack{1\le l\le k\\ l\ne d}}\beta_lx^l\biggr),
\]
and seek the conjectured mean value estimate
\[
\int_{[0,1)^{k-1}}|f_{k,d}(\boldsymbol \beta ;X)|^{2s}\d\boldsymbol \beta \ll X^{s+\varepsilon}+X^{2s-(k^2+k-2d)/2}.
\]
This scenario is explored in \cite[Theorem 1.2]{BW2017}, where one finds that the strategy of this memoir does not 
achieve optimal estimates for $d>1$. The key difficulty is evident in our proof of Theorem \ref{theorem1.1}. By 
orthogonality, the missing monomial $x^d$ in the exponential sum $f_{k,d}(\boldsymbol \beta;X)$ is associated to an 
omitted diagonal equation in a system analogous to \eqref{1.3}. This equation may be recovered by introducing an 
auxiliary variable $h$ having size of order $X^d$. The reader may inspect the definition \eqref{2.1} of the exponential 
sum $g_k(\boldsymbol \beta;X)$ below to surmise how it is that this auxiliary variable makes its presence known in 
our proof. In order to obtain sharp mean value estimates for the exponential sum $f_{k,d}(\boldsymbol \beta;X)$, one 
would need to obtain $(d+1)$-th power cancellation in mean values involving the analogue of 
$g_k(\boldsymbol \alpha;X)$, a prospect inaccessible in pure mean value estimates of this sum when $d>1$.\par

We organise this paper as follows. First, in \S2, we derive a mean value estimate for an auxiliary bivariate exponential 
sum by applying our recent work \cite{Woo2026}. Here, in the rational integer case of our main theorem, earlier work 
of Guo and Zorin-Kranich \cite{GZK2020} would suffice as input. However, we seek results applicable also in number 
fields and function fields, and for this reason the appeal to our work \cite{Woo2026} is critical. A detail worth noting is 
that care must be taken to provide mean value estimates valid for odd as well as even order moments, so that some 
additional attention is required en route. In \S3 we apply this auxiliary mean value to prove Theorem \ref{theorem1.1} 
for $s\le (k^2+k-2)/2$, deferring to \S4 an account of the sharp $\varepsilon$-free estimate available for 
$s>(k^2+k-2)/2$. In \S\S5 and 6, respectively, we explain how estimates analogous to that of Theorem 
\ref{theorem1.1} may be derived in the setting of number fields and function fields. Finally, in \S7, we make additional 
comments concerning the cubic case of our main theorem, a topic with an extensive history that provides a canvas to 
explore discrete restriction variants of our results.\par

Some notational conventions assist in our exposition. We make extensive use of vector notation. Thus, for 
example, we write $1\le \mathbf x\le X$ to denote that every coordinate $x_i$ of $\mathbf x$ satisfies $1\le x_i\le X$. 
Also, given a situation in which the parameter $\varepsilon$ has not already been fixed, in any statement involving 
$\varepsilon$, we implicitly assert that the statement holds for each $\varepsilon>0$. Implicit constants in 
Vinogradov's notation $\ll $ and $\gg $ may depend on $\varepsilon$.\par

Work on this paper was conducted while the author was supported by NSF grants DMS-2001549 and DMS-2502625, 
and by Simons Fellowship in Mathematics SFM-00011955. The author is grateful to the Institute for Advanced Study, 
Princeton, for hosting his sabbatical, during which period this paper was completed.

\section{An auxiliary mean value theorem}
The proof of Theorem \ref{theorem1.1} makes essential use of a mean value estimate for a certain auxiliary 
exponential sum not of translation-dilation invariant type. In this section we relate this mean value to a more 
complicated enveloping mean value that is fully traslation-dilation invariant. We begin with some notation. When 
$\boldsymbol \alpha,\boldsymbol \beta\in \mathbb R^k$, define the polynomials
\[
\varphi_k(z;\boldsymbol \alpha)=\sum_{i=1}^k\alpha_iz^i\quad \text{and}\quad 
\psi_k(z,h;\boldsymbol \beta)=\sum_{i=1}^k\beta_ihz^{i-1}.
\]
We consider the exponential sums
\begin{align}
g_k(\boldsymbol \beta;X)&=\sum_{1\le z\le X}\biggl| \sum_{1\le h\le X}e(\psi_k(z,h;\boldsymbol \beta))
\biggr|^2,\label{2.1}\\
b_k(\boldsymbol \alpha,\boldsymbol \beta;X)&=\sum_{1\le z\le X}\sum_{1\le h\le X}e(\varphi_k(z;\boldsymbol \alpha)
+\psi_k(z,h;\boldsymbol \beta)),\label{2.2}
\end{align}
and the respective associated mean values
\begin{align}
A_{s,k}(X)&=\int_{[0,1)^k}|g_k(\boldsymbol \beta ;X)|^s\d\boldsymbol \beta ,\label{2.3}\\
B_{s,k}(X)&=\int_{[0,1)^k}\int_{[0,1)^k} |b_k(\boldsymbol \alpha,\boldsymbol \beta;X)|^{2s}\d\boldsymbol \alpha 
\d\boldsymbol \beta .\label{2.4}
\end{align}

The goal of this section is the mean value estimate contained in the following lemma.

\begin{lemma}\label{lemma2.1}
When $k\ge 2$ and $s>0$, one has
\[
A_{s,k}(X)\ll X^{s+\varepsilon}\bigl( X^s+X^{2s-k(k+1)/2}\bigr).
\]
\end{lemma}

\begin{proof} Write $t=k(k+1)/2$. We begin by considering the special case of the mean value \eqref{2.4} in which 
$s=t$. We begin by rewriting $g_k(\boldsymbol \beta;X)$ in terms of $b_k(\boldsymbol \alpha,\boldsymbol \beta;X)$. 
Thus, on recalling \eqref{2.1} and \eqref{2.2}, it follows via orthogonality that
\[
g_k(\boldsymbol \beta ;X)=\sum_{1\le z\le X}
\biggl| \sum_{1\le h\le X}e(\psi_k(z,h;\boldsymbol \beta))\biggr|^2
=\int_{[0,1)^k}|b_k(\boldsymbol \alpha,\boldsymbol \beta;X)|^2\d\boldsymbol \alpha .
\]
By applying H\"older's inequality, we therefore deduce from \eqref{2.3} and \eqref{2.4} that
\begin{align}
A_{t,k}(X)&=\int_{[0,1)^k}\biggl( \int_{[0,1)^k}|b_k(\boldsymbol \alpha,\boldsymbol \beta;X)|^2\d\boldsymbol \alpha 
\biggr)^t\d\boldsymbol \beta \notag \\
&\le \int_{[0,1)^k}\int_{[0,1)^k}|b_k(\boldsymbol \alpha, \boldsymbol \beta ;X)|^{2t}\d\boldsymbol \alpha 
\d\boldsymbol \beta =B_{t,k}(X).\label{2.5}
\end{align}

\par By orthogonality, the conclusion of \cite[Corollary 1.4]{Woo2026} shows that
\begin{equation}\label{2.6}
B_{t,k}(X)\ll X^\varepsilon \bigl( X^{2t}+X^{4t-k(k+1)}\bigr) \ll X^{2t+\varepsilon}.
\end{equation}
The same conclusion may also be extracted from \cite[Theorem 1]{GZK2020} with careful attention to notation. Since 
we have in mind generalisations of our main conclusions to the setting of number fields and function fields, we 
emphasise the appeal to \cite{Woo2026} over \cite{GZK2020}. In this way, we deduce from \eqref{2.5} that 
$A_{t,k}(X)\ll X^{2t+\varepsilon}$. This estimate suffices to obtain the upper bound claimed by the lemma in 
full. Indeed, by applying H\"older's inequality to \eqref{2.3}, one sees that for $0<s\le t$, one has
\[
A_{s,k}(X)\le \bigl( A_{t,k}(X)\bigr)^{s/t}\ll (X^{2t+\varepsilon})^{s/t}\ll X^{2s+\varepsilon}.
\]
Meanwhile, by employing the trivial estimate $|g_k(\boldsymbol \beta;X)|\le X^3$, one sees for $s>t$ that
\[
A_{s,k}(X)\le (X^3)^{s-t}A_{t,k}(X)\ll (X^3)^{s-t}X^{2t+\varepsilon}\ll X^{3s-t+\varepsilon}.
\]
Thus, whenever $s>0$, one has the upper bound $A_{s.k}(X)\ll X^\varepsilon (X^{2s}+X^{3s-t})$, and the desired 
conclusion follows on recalling that $t=k(k+1)/2$.
\end{proof}

\section{Shifting to Vinogradov's mean value theorem}
Our strategy for proving Theorem \ref{theorem1.1} is based on the proof of the special case in which $k=3$ and 
$s\le 4$ established in \cite[Theorem 1.3]{Woo2015}, and extended in \cite[Theorem 1.1]{BW2017} to handle those 
situations in which $k\ge 3$ and $s\le (k^2-1)/2$. We adapt this approach so as to avoid odd moments of auxiliary 
exponential sums lacking a Diophantine interpretation. We begin by introducing additional notation. First, we 
introduce the exponential sum
\begin{equation}\label{3.1}
F_k(\boldsymbol \alpha;X)=\sum_{1\le x\le X}e(\alpha_1x+\ldots +\alpha_kx^k),
\end{equation}
that extends the exponential sum $f_k(\boldsymbol \beta;X)$ defined in \eqref{1.1}. Note that 
$F_k(\boldsymbol \alpha;X)$ includes a linear term $\alpha_1x$ in the argument. We consider both Vinogradov's 
mean value
\begin{equation}\label{3.2}
J_{s,k}(X)=\int_{[0,1)^k}|F_k(\boldsymbol \alpha ;X)|^{2s}\d\boldsymbol \alpha ,
\end{equation}
and the closely associated mean value
\begin{equation}\label{3.3}
I_{s,k}(X)=\int_{[0,1)^{k-1}}|f_k(\boldsymbol \beta ;X)|^{2s}\d\boldsymbol \beta .
\end{equation}
We also have need of the exponential sum $g_k(\boldsymbol \beta;X)$ defined in \eqref{2.1}, as well as
\begin{equation}\label{3.4}
g^*_{k,s}(\boldsymbol \alpha;z;X)=\sum_{1\le h\le 2sX}e(h\alpha_1+2hz\alpha_2+\ldots +khz^{k-1}\alpha_k). 
\end{equation}

\begin{lemma}\label{lemma3.1}
When $s$ is a natural number, one has
\[
I_{s,k}(X)\ll X^{-2}\int_{[0,1)^k}|F_k(\boldsymbol \alpha;2X)|^{2s}g_k(\boldsymbol \alpha;2sX)
\d\boldsymbol \alpha .
\]
\end{lemma}

\begin{proof} Define $\delta_j$ to be $1$ when $j=1$, and $0$ otherwise. Then by interpreting the mean value 
$I_{s,k}(X)$ in terms of the system \eqref{1.3} via orthogonality, we see that $I_{s,k}(X)$ counts the number of integral 
solutions of the system of equations
\begin{equation}\label{3.5}
\sum_{i=1}^s(x_i^j-y_i^j)=\delta_jh\quad (1\le j\le k),
\end{equation}
with $1\le \mathbf x,\mathbf y\le X$ and $|h|\le sX$. Here, the range for $h$ has been chosen sufficiently large that 
the equation with $j=1$ in \eqref{3.5} is automatically satisfied. Denote by $I_{s,k}^*(X)$ the number of solutions of 
the system of equations
\begin{equation}\label{3.6}
\sum_{i=1}^s(x_i^j-y_i^j)=\delta_j(h_1-h_2)\quad (1\le j\le k),
\end{equation}
with $1\le \mathbf x,\mathbf y\le X$ and $1\le \mathbf h\le 2sX$. Observe that, when $|h|\le sX$, there are at least 
$X$ choices for $h_1$ and $h_2$ with $1\le \mathbf h\le 2sX$ satisfying the equation $h=h_1-h_2$. Thus,
\begin{equation}\label{3.7}
I_{s,k}(X)\le X^{-1}I_{s,k}^*(X).
\end{equation}

\par We next consider a shift of each variable $x_i,y_i$ in \eqref{3.6} by an integer $z$ with $1\le z\le X$. By the 
binomial theorem, for each shift $z$, one finds that $(\mathbf x,\mathbf y)$ is a solution of \eqref{3.6} if and only if 
it is also a solution of the system
\[
\sum_{i=1}^s\bigl( (x_i+z)^j-(y_i+z)^j\bigr) =j(h_1-h_2)z^{j-1}\quad (1\le j\le k).
\]
It follows that for each fixed integer $z$ with $1\le z\le X$, the mean value $I_{s,k}^*(X)$ is bounded above by the 
number of integral solutions of the system
\[
\sum_{i=1}^s (u_i^j-v_i^j)=j(h_1-h_2)z^{j-1}\quad (1\le j\le k),
\]
with $1\le \mathbf u,\mathbf v\le 2X$ and $1\le \mathbf h\le 2sX$. By averaging over these choices for $z$ and 
applying orthogonality, therefore, we deduce from \eqref{3.1}, \eqref{3.4} and \eqref{2.1} that
\begin{align*}
I_{s,k}^*(X)&\le X^{-1}\sum_{1\le z\le X}\int_{[0,1)^k}|F_k(\boldsymbol \alpha;2X)|^{2s}
|g^*_{k,s}(\boldsymbol \alpha ;z;X)|^2\d\boldsymbol \alpha \\
&\le X^{-1}\int_{[0,1)^k}|F_k(\boldsymbol \alpha ;2X)|^{2s}g_k(\boldsymbol \alpha;2sX)\d\boldsymbol \alpha .
\end{align*}
The desired conclusion is now immediate from \eqref{3.7}.
\end{proof}

In advance of our proof of Theorem \ref{theorem1.1}, we pause to record the mean value
\begin{equation}\label{3.8}
J_{s,k}(X)\ll X^{s+\varepsilon}+X^{2s-k(k+1)/2}
\end{equation}
that, in view of the definitions \eqref{3.1} and \eqref{3.2}, has already been reported in \eqref{1.4}. We begin by 
establishing a mean value very slightly weaker than that given in Theorem \ref{theorem1.1}.

\begin{lemma}\label{lemma3.2}
Suppose that $k\ge 2$ and $s>0$. Then for each $\varepsilon>0$, one has
\[
I_{s,k}(X)\ll X^\varepsilon (X^s+X^{2s-(k^2+k-2)/2}).
\]
\end{lemma}

\begin{proof} We again write $t=k(k+1)/2$. Observe that from Lemma \ref{lemma3.1}, one has
\[
I_{t-1,k}(X)\ll X^{-2}\int_{[0,1)^k}|F_k(\boldsymbol \alpha ;2X)|^{2t-2}g_k(\boldsymbol \alpha;(2t-2)X)\d\boldsymbol 
\alpha .
\]
Hence, on recalling \eqref{2.3} and \eqref{3.2}, an application of H\"older's inequality combines with \eqref{3.8} and 
Lemma \ref{lemma2.1} to yield the bound
\begin{align*}
I_{t-1,k}(X)&\ll X^{-2} \bigl( J_{t,k}(2X)\bigr)^{1-1/t}\bigl( A_{t,k}((2t-2)X)\bigr)^{1/t}\\
&\ll X^{-2}(X^{t+\varepsilon})^{1-1/t}(X^{2t+\varepsilon})^{1/t}.
\end{align*}
We therefore conclude that $I_{t-1,k}\ll X^{t-1+\varepsilon}$, a bound that suffices to deliver the conclusion of the 
lemma in full, as we now explain. Indeed, when $0<s\le t-1$, one applies H\"older's inequality within \eqref{3.3} to 
obtain
\[
I_{s,k}(X)\le \bigl( I_{t-1,k}(X)\bigr)^{s/(t-1)}\ll X^{s+\varepsilon}.
\]
Meanwhile, when $s>t-1$, the trivial estimate $|f_k(\boldsymbol \beta;X)|\le X$ yields
\[
I_{s,k}(X)\le X^{2s-(2t-2)}I_{t-1,k}(X)\ll X^{2s-(t-1)+\varepsilon}.
\]
Thus, for any $s>0$, one has $I_{s,k}(X)\ll X^\varepsilon (X^s+X^{2s-(t-1)})$, and the desired conclusion follows on 
recalling that $t=k(k+1)/2$.
\end{proof}

\section{Sharp estimates beyond the critical range}
The conclusion of Lemma \ref{lemma3.2} suffices to establish Theorem \ref{theorem1.1} when $s\le (k^2+k-2)/2$. 
However, in the complementary range $s>(k^2+k-2)/2$, the desired upper bound is missed by a factor of 
$X^\varepsilon$. It is the purpose of this section to repair this slight deficiency.

\begin{lemma}\label{lemma4.1} Suppose that $k\ge 2$ and $s>(k^2+k-2)/2$. Then one has
\[
I_{s,k}(X)\ll X^{2s-(k^2+k-2)/2}.
\]
\end{lemma}

\begin{proof} We apply the Hardy-Littlewood method. In order to streamline our discussion, write 
$u=(k^2+k-2)/2$, and put $\delta=(s-u)/k$. Then the hypotheses of the statement of the lemma ensure that 
$\delta>0$. Write $L=X^{1/(4k)}$. Then, when
\[
1\le q\le L,\quad 0\le \mathbf a\le q\quad \text{and}\quad (q,a_2,\ldots ,a_k)=1,
\]
we define the major arc $\mathfrak M(q,\mathbf a)$ by
\[
\mathfrak M(q,\mathbf a)=\{ \boldsymbol \alpha\in [0,1)^{k-1}:\text{$|\alpha_j-a_j/q|\le LX^{-j}$ $(2\le j\le k)$}\}.
\]
The arcs $\mathfrak M(q,\mathbf a)$ are disjoint, as one may easily check. Let $\mathfrak M$ denote the union of 
these arcs $\mathfrak M(q,\mathbf a)$, and put $\mathfrak m=[0,1)^{k-1}\setminus \mathfrak M$.\par

We recall the definition \eqref{1.1} of $f_k(\boldsymbol \alpha ;X)$, and when 
$\boldsymbol \alpha\in \mathfrak M(q,\mathbf a)\subseteq \mathfrak M$, we put
\[
V(\boldsymbol \alpha ;q,\mathbf a)=q^{-1}S(q,\mathbf a)I(\boldsymbol \alpha -\mathbf a/q;X),
\]
where
\[
S(q,\mathbf a)=\sum_{r=1}^qe\bigl( (a_2r^2+\ldots +a_kr^k)/q\bigr)
\]
and
\[
I(\boldsymbol \beta ;X)=\int_0^Xe(\beta_2\gamma^2+\ldots +\beta_k\gamma^k)\d\gamma .
\]  
We then define $V(\boldsymbol \alpha)$ to be $V(\boldsymbol \alpha;q;\mathbf a)$ when 
$\boldsymbol \alpha\in \mathfrak M(q,\mathbf a)\subseteq \mathfrak M$, and to be $0$ otherwise.\par

The minor arc contribution is estimated by an argument akin to that associated with \cite[equation (7.1)]{Woo2017}. 
Thus, for each $\varepsilon >0$, one has
\[
\sup_{\boldsymbol \alpha\in \mathfrak m}|f_k(\boldsymbol \alpha;X)|\ll X^{1-\tau+\varepsilon},
\]
where $\tau^{-1}=8k^2$. Since $s-u=k\delta>0$, one finds via Lemma \ref{lemma3.2} that
\begin{align*}
\int_{\mathfrak m}|f_k(\boldsymbol \alpha;X)|^{2s}\d\boldsymbol \alpha &\ll 
\biggl( \sup_{\boldsymbol \alpha \in \mathfrak m}|f_k(\boldsymbol \alpha;X)|\biggr)^{2s-2u}
\int_{[0,1)^{k-1}}|f_k(\boldsymbol \alpha ;X)|^{2u}\d\boldsymbol \alpha \\
&\ll (X^{1-\tau})^{2k\delta}X^{u+\varepsilon},
\end{align*}
whence
\begin{equation}\label{4.1}
\int_{\mathfrak m}|f_k(\boldsymbol \alpha;X)|^{2s}\d\boldsymbol \alpha =o(X^{2s-u}).
\end{equation}

\par Meanwhile, as in the proof of \cite[Lemma 7.1]{Woo2017}, one finds that when 
$\boldsymbol \alpha \in \mathfrak M(q,\mathbf a)\subseteq \mathfrak M$, one has 
$f_k(\boldsymbol \alpha;X)-V(\boldsymbol \alpha;q,\mathbf a)\ll L^2$, whence
\[
|f_k(\boldsymbol \alpha ;X)|^{2s}-|V(\boldsymbol \alpha ;q,\mathbf a)|^{2s}\ll L^2X^{2s-1}.
\]
Since $\text{mes}(\mathfrak M)\ll L^{2k-1}X^{-u}$, we deduce that
\begin{equation}\label{4.2}
\int_{\mathfrak M}|f_k(\boldsymbol \alpha;X)|^{2s}\d\boldsymbol \alpha -
\int_{\mathfrak M}|V(\boldsymbol \alpha)|^{2s}\d\boldsymbol \alpha \ll (L^{2k+1}/X)X^{2s-u}=o(X^{2s-u}).
\end{equation}

\par Next, on recalling the definitions of $V(\boldsymbol \alpha)$ and $\mathfrak M$, we find that
\begin{equation}\label{4.3}
\int_{\mathfrak M}|V(\boldsymbol \alpha)|^{2s}\d\boldsymbol \alpha \le \mathfrak SJ,
\end{equation}
where
\[
\mathfrak S=\sum_{q=1}^\infty \underset{(q,a_2,\ldots ,a_k)=1}{\sum_{a_2=1}^q\cdots \sum_{a_k=1}^q}
|q^{-1}S(q,\mathbf a)|^{2s}\quad \text{and}\quad J=\int_{\mathbb R^{k-1}}|I(\boldsymbol \beta ;X)|^{2s}
\d\boldsymbol \beta .
\]
When $(q,a_2,\ldots ,a_k)=1$, we find from \cite[Theorem 7.1]{Vau1997} that 
$S(q,\mathbf a)\ll q^{1-1/k+\varepsilon}$. Then, since $2u/k=k+1-2/k$, we deduce that
\begin{equation}\label{4.4}
\mathfrak S\ll \sum_{q=1}^\infty q^{k-1+\varepsilon}(q^{-1/k})^{2u+2k\delta}\ll \sum_{q=1}^\infty q^{-1-\delta}\ll 1.
\end{equation}
Likewise, from \cite[Theorem 7.3]{Vau1997}, one has
\[
I(\boldsymbol \beta;X)\ll X\bigl( 1+|\beta_2|X^2+\ldots +|\beta_k|X^k\bigr)^{-1/k}\le X\prod_{j=2}^k 
\bigl( 1+|\beta_j|X^j\bigr)^{-1/(k^2-k)}.
\]
Since
\[
\frac{2s}{k^2-k}>\frac{k^2+k-2}{k^2-k}=1+\frac{2}{k},
\]
we therefore find that
\begin{equation}\label{4.5}
J\ll X^{2s}\prod_{j=2}^k\biggl( \int_{-\infty}^\infty (1+|\beta_j|X^j)^{-1-2/k}\d\beta_j\biggr) \ll X^{2s-u}.
\end{equation}

\par On substituting \eqref{4.4} and \eqref{4.5} into \eqref{4.3}, and thence into \eqref{4.2}, we infer that
\[
\int_{\mathfrak M}|f_k(\boldsymbol \alpha;X)|^{2s}\d\boldsymbol \alpha=
\int_{\mathfrak M}|V(\boldsymbol \alpha)|^{2s}\d\boldsymbol \alpha +o(X^{2s-u})\ll X^{2s-u},
\]
whence by \eqref{3.3} and \eqref{4.1}, we have
\[
I_{s,k}(X)=\int_{\mathfrak M}|f_k(\boldsymbol \alpha;X)|^{2s}\d\boldsymbol \alpha +
\int_{\mathfrak m}|f_k(\boldsymbol \alpha;X)|^{2s}\d\boldsymbol \alpha \ll X^{2s-u}.
\]
The conclusion of the lemma follows on recalling that $u=(k^2+k-2)/2$.
\end{proof}

It is a triviality now to complete the proof of Theorem \ref{theorem1.1}.

\begin{proof}[The proof of Theorem \ref{theorem1.1}] The conclusion of Theorem \ref{theorem1.1} follows for 
$s\le (k^2+k-2)/2$ from Lemma \ref{lemma3.2}, and for $s>(k^2+k-2)/2$ from Lemma \ref{lemma4.1}.
\end{proof}

We remark that, with a little additional effort, one would obtain an asymptotic formula for $I_{s,k}(X)$ when 
$s>(k^2+k-2)/2$ by adapting the proof of Lemma \ref{lemma4.1}.

\begin{theorem}\label{theorem4.2}
There are positive numbers $C=C(s,k)$ and $\delta=\delta(s,k)$ having the property that, whenever $s>(k^2+k-2)/2$, 
one has
\[
I_{s,k}(X)=CX^{2s-(k^2+k-2)/2}+O(X^{2s-(k^2+k-2)/2 -\delta}).
\]
\end{theorem}

We leave the proof of this conclusion to the enthusiastic reader, noting that the proof of \cite[Theorem 7.4]{Vau1997} 
provides a pedestrian route plan for the required argument.

\section{A sharp mean value estimate in number fields}
We begin by describing the infrastructure required to describe mean value estimates in number fields. Consider an 
algebraic extension $K$ of degree $n$ over $\mathbb Q$. Associated with $K$ are the real conjugate fields $K^{(l)}$ 
$(1\le l\le n_1)$, and the pairs of complex conjugate fields $K^{(m)}$ and $K^{(m+n_2)}$ $(n_1+1\le m\le n_1+n_2)$. 
Here, one has $n=n_1+2n_2$. Write $\mathfrak O_K$ for the ring of integers of $K$, and fix a basis 
$\Omega=\{ \omega_1,\ldots ,\omega_n\}$ for $\mathfrak O_K$ over $\mathbb Z$. When $X$ is a large positive 
number, we define the box $\mathcal B(X)\subset \mathfrak O_K$, with $\text{card}(\mathcal B(X))\asymp X$, by 
taking
\[
\mathcal B(X)=\bigl\{ r_1\omega_1+\ldots +r_n\omega_n: 
\text{$r_i\in \bigl[-\tfrac{1}{2}X^{1/n},\tfrac{1}{2}X^{1/n}\bigr)\cap \mathbb Z$ $(1\le i\le n)$}\bigr\} .
\]

\par When $\gamma\in K$, denote by $\gamma^{(i)}$ the conjugates of $\gamma$, with $\gamma^{(i)}\in K^{(i)}$ 
$(1\le i\le n)$, and the associated trace map 
${\rm{Tr}}_{K/\mathbb Q}(\gamma)=\gamma^{(1)}+\ldots +\gamma^{(n)}$. Next, when 
$\boldsymbol \gamma\in K^n$ and $\boldsymbol \theta\in \mathbb R^n$, and
\[
\lambda(\boldsymbol \theta)=\theta_1\gamma_1+\ldots +\theta_n\gamma_n,
\]
we define
\[
\lambda^{(i)}(\boldsymbol \theta)=\theta_1\gamma_1^{(i)}+\ldots +\theta_n\gamma_n^{(i)}\quad (1\le i\le n).
\]
The trace map on $K$ then extends to the map defined by
\[
{\rm{Tr}}_{K/\mathbb Q}(\lambda(\boldsymbol \theta))=\lambda^{(1)}(\boldsymbol \theta)+\ldots +\lambda^{(n)}
(\boldsymbol \theta),
\]
and we define
\[
E(\lambda(\theta))=e\bigl( {\rm{Tr}}_{K/\mathbb Q}(\lambda(\boldsymbol \theta))\bigr) .
\]

\par Next, let $\mathfrak d^{-1}$ denote the inverse different, so that
\[
\mathfrak d^{-1}=\{ \eta\in K:\text{${\rm{Tr}}_{K/\mathbb Q}(\eta\xi)\in \mathbb Z$ for all $\xi\in \mathfrak O_K$}\}.
\]
There is a basis ${\rm P}=\{\rho_1,\ldots ,\rho_n\}$ of $\mathfrak d^{-1}$ dual to $\Omega$ having the property that
\[
{\rm{Tr}}_{K/\mathbb Q}(\rho_i\omega_j)=\begin{cases} 1,&\text{when $i=j$,}\\
0,&\text{when $i\ne j$.}\end{cases}
\]
As a special case of the linear form $\lambda(\boldsymbol \theta)$ defined above, we define 
$\alpha=\alpha(\boldsymbol \theta)$ by
\[
\alpha(\boldsymbol \theta)=\theta_1\rho_1+\ldots +\theta_n\rho_n.
\]
When such a form occurs in an $n$-fold integral, we abbreviate the $n$-fold differential 
${\rm d}\theta_1\cdots \d\theta_n$ to ${\rm d}\alpha$. We also write $\mathbb T$ for $[0,1)^n$. In this notation, 
we have the orthogonality relation
\[
\int_{\mathbb T}E(\alpha\gamma)\d\alpha =\begin{cases}1,&\text{when $\gamma=0$,}\\
0,&\text{when $\gamma\in \mathfrak O_K\setminus \{0\}$.}\end{cases}
\]
The analogue of Theorem \ref{theorem1.1} for number fields may now be announced.

\begin{theorem}\label{theorem5.1}
Suppose that $k\ge 2$ and $s>0$. Then for each $\varepsilon>0$, one has
\[
\int_{\mathbb T^{k-1}}\biggl| \sum_{x\in \mathcal B(X)}E(\alpha_2x^2+\ldots +\alpha_kx^k)\biggr|^{2s}
\d\boldsymbol \alpha \ll X^\varepsilon \bigl( X^s+X^{2s-(k^2+k-2)/2}\bigr).
\]
\end{theorem}

\begin{proof} In the number field setting, one may proceed just as in the arguments described in \S\S2 and 3. The 
conclusion of \cite[Theorem 11.2]{Woo2026} supplies the analogue of \eqref{2.6} needed to prove the number field 
version of Lemma \ref{lemma2.1}. Moreover, the analogue of the bound \eqref{3.8} is supplied by 
\cite[Corollary 15.3]{Woo2019}. Thus, the argument of \S3 applies mutatis mutandis.
\end{proof}

Additional work is needed to obtain the analogue of the \S4 arguments capable of delivering an $\varepsilon$-free 
version of Theorem \ref{theorem5.1} for $s>(k^2+k-2)/2$. This we leave as an exercise for the enthusiastic reader.

\section{A function field variant}
We again initiate proceedings in function fields by introducing additional infrastructure. Consider a finite field 
$\mathbb F_q$, with $q=p^l$, of characteristic $p>k$. Associated with the polynomial ring 
$\mathbb O=\mathbb F_q[t]$ is the field of fractions $\mathbb K =\mathbb F_q(t)$. We take as our basic 
parameter $X$, a sufficiently large natural number, and we put
\[
\mathbb O_X=\{ \nu\in \mathbb F_q[t]:\text{deg}\, \nu\le X\}.
\]
Write $\mathbb K_\infty=\mathbb F_q((1/t))$ for the completion of $\mathbb F_q(t)$ at $\infty$. Each element 
$\alpha \in \mathbb K_\infty\setminus \{0\}$ may be written in the shape $\alpha=\sum_{i\le n}a_it^i$, for some 
$n\in \mathbb Z$, with $a_i=a_i(\alpha)\in \mathbb F_q$ for $-\infty <i\le n$ and $a_n\ne 0$. We define 
$\text{ord}\, \alpha=n$, adopt the convention that $\text{ord}\, 0=-\infty$, and introduce the compact additive 
subgroup $\mathbb T=\{ \alpha \in \mathbb K_\infty: \text{ord}\, \alpha<0\}$ of $\mathbb K_\infty$. We may 
normalise any Haar measure ${\rm d}\alpha$ on $\mathbb K_\infty$ in such a manner that 
$\int_\mathbb T\d\alpha =1$.\par

We next introduce the non-trivial character $e_q:\mathbb F_q\rightarrow \mathbb C^\times$, defined for 
$a\in \mathbb F_q$ by taking $e_q(a)=e({\rm tr}(a)/p)$, where ${\rm tr}(a)=a+a^p+\ldots +a^{p^{l-1}}$. Define 
$\text{res}(\alpha)$ to be $a_{-1}(\alpha)$ when $\text{ord}\, \alpha\ge -1$, and $0$ otherwise. Then, this 
character $e_q(\,\cdot\,)$ induces a map $e:\mathbb K_\infty \rightarrow \mathbb C^\times$, defined by 
$e(\alpha)=e_q(\text{res}\, \alpha)$ for $\alpha\in \mathbb K_\infty$. One then has
\[
\int_\mathbb T e(h\alpha)\d\alpha =\begin{cases}1,&\text{when $h=0$,}\\
0,&\text{when $h\in \mathbb F_q[t]\setminus \{0\}$.}\end{cases}
\]

We now announce our analogue of Theorem \ref{theorem1.1} for function fields.

\begin{theorem}\label{theorem6.1} Suppose that $k\ge 2$ and $s>0$. Then, whenever 
${\rm ch}(\mathbb F_q)>k$ and $\varepsilon>0$,
\[
\int_{\mathbb T^{k-1}}\biggl| \sum_{x\in \mathbb O_X}e(\alpha_2x^2+\ldots +\alpha_kx^k)\biggr|^{2s}\d\boldsymbol 
\alpha \ll (q^X)^\varepsilon \bigl( (q^X)^s+(q^X)^{2s-(k^2+k-2)/2}\bigr) .
\]
\end{theorem}

\begin{proof} We follow \S\S2 and 3, modified mutatis mutandis, to the function field setting. The conclusion of 
\cite[Theorem 12.1]{Woo2026} supplies the analogue of \eqref{2.6} needed to prove the analogue of Lemma 
\ref{lemma2.1}, the analogue of \eqref{3.8} being supplied by \cite[Corollary 17.3]{Woo2019}.
\end{proof}

We again leave as an exerise the additional work required to obtain the analogue of the arguments in \S4 that deliver 
an $\varepsilon$-free version of Theorem \ref{theorem6.1} for $s>(k^2+k-2)/2$.

\section{Some comments on the cubic case of Theorem \ref{theorem1.1}}
We now address the special case of Theorem \ref{theorem1.1} in which $k=3$. Here, we have
\[
\int_0^1 \!\! \int_0^1 \biggl| \sum_{1\le x\le X}e(\alpha x^3+\beta x^2)\biggr|^{2s}\d\alpha\d\beta \ll 
X^{s+\varepsilon}+X^{2s-5},
\]
and when $s=5$, in particular, the upper bound
\begin{equation}\label{7.1}
\int_0^1 \!\! \int_0^1 \biggl| \sum_{1\le x\le X}e(\alpha x^3+\beta x^2)\biggr|^{10}\d\alpha\d\beta \ll 
X^{5+\varepsilon}.
\end{equation}
The mean value on the left hand side of \eqref{7.1} has been the subject of intensive investigation by the author 
over the past 35 years. Early work \cite{Woo1990} and \cite{Woo1991} was restricted to exponential sums over 
smooth numbers, obtaining analogues of \eqref{7.1} in which the exponent $5$ is replaced by $35/6$, with 
subsequent work \cite{Woo1998} improving the latter exponent to $17/3$. More recently, 
following the author's resolution \cite{Woo2016} of the main conjecture in the cubic case of Vinogradov's mean value 
theorem, the upper bound \eqref{7.1} was established with the exponent $5$ replaced by $16/3$. The latter work was 
employed by the author \cite{Woo2015} to establish a quantitative Hasse principle for pairs of diagonal Diophantine 
equations, one cubic and one quadratic, in $11$ or more variables. This conclusion was much extended in work of 
Brandes \cite{Bra2017} concerning systems of (many) diagonal Diophantine cubic and quadratic equations. The sharp 
upper bound \eqref{7.1} should foster improvements in the latter results so far as the number of variables is 
concerned.\par

One challenge that remains unresolved concerns a discrete restriction analogue of the upper bound \eqref{7.1}. Let 
$(a_n)$ denote a sequence of complex numbers. It is plausible that
\begin{equation}\label{7.2}
\int_0^1 \!\! \int_0^1 \biggl| \sum_{1\le n\le X}a_ne(\alpha n^3+\beta n^2)\biggr|^{10}\d\alpha\d\beta \ll 
X^\varepsilon \biggl( \sum_{1\le n\le X}|a_n|^2\biggr)^5.
\end{equation}
However, the methods of this paper establish such a bound only when the $a_n$ are more-or-less 
$1$-bounded and not too thinly supported. The sharpest available unconditional estimate of this flavour would 
appear to be \cite[Corollary 1.3]{Woo2017}, which gives the bound
\[
\int_0^1 \!\! \int_0^1 \biggl| \sum_{1\le n\le X}a_ne(\alpha n^3+\beta n^2)\biggr|^{12}\d\alpha\d\beta \ll 
X^{1+\varepsilon} \biggl( \sum_{1\le n\le X}|a_n|^2\biggr)^6.
\]
Consequently, the proof of \eqref{7.2} apparently remains a significant challenge for aficionados.

\bibliographystyle{amsbracket}
\providecommand{\bysame}{\leavevmode\hbox to3em{\hrulefill}\thinspace}

\end{document}